\documentclass[11pt]{article}
\usepackage{amsmath,amssymb,amsthm}
\usepackage{txfonts}
\usepackage{graphicx}
\usepackage{wrapfig}
\usepackage{pifont}
\usepackage{color}

\newtheorem{dfn}{Definition}[section]
\newtheorem{thm}[dfn]{Theorem}
\newtheorem{lem}[dfn]{Lemma}

\makeatletter

\def\beq{\begin{eqnarray*}}
\def\eeq{\end{eqnarray*}}

\def\beqn{\begin{equation}}
\def\eeqn{\end{equation}}

\def\1{{\bf 1}}

\def\v2{\vskip2mm}

\def\tst12{{\textstyle\frac12}}

\begin{document}
\date{}
\title{Last zero time or Maximum time of the winding number of Brownian motions}
\author{Izumi Okada}
\maketitle
\begin{abstract}
In this paper we consider  the winding number, $\theta(s)$, of planar Brownian motion and  study asymptotic behavior  of the process of  the maximum time, the time when $\theta(s)$ attains  the maximum  in the interval  $0\le s \le t$.  We find the limit law of its logarithm with a suitable normalization factor  and   the upper growth rate of  the maximum time process  itself.  We also show that the process of  the last zero time of  $\theta(s)$ in $[0,t]$ has the same law as the maximum time process.
\end{abstract}
\section{Introduction and Main results}
In this paper we seek for an analogue of the arcsine law of the linear Brownian motion for the argument of a complex Brownian motion $\{W(t)= W_1(t) +i W_2(t): t\ge0\}$ started at $W(0)= (1,0)$. 
Skew-product representation tells us that 
there exist two independent linear Brownian motions $\{B(t): t\ge0\}$  and $\{\hat{B}(t): t\ge0\}$ such that 
\begin{align}\label{skew}
W(t)=\exp (\hat{B}(H(t))+iB(H(t))) \text{ for all }t\ge0, 
\end{align}
where 
\begin{align*}
H(t)=\int_0^t \frac{ds}{|W(s)|^2}=\inf\{u\ge0: \int_0^u \exp(2\hat{B}(s))ds>t\},
\end{align*}
which entails that $B$ is independent of $|W|$ and hence of $H$, while   $\log |W|$ is time change of $\hat{B}$ (cf. e.g., \cite{Mo}, Theorem $7.26$). 

We let $\theta(t)=B(H(t))$ so that $\theta(t)=\arg W(t)$, which we call the winding number. 
Without loss of  generality  we suppose $\theta(0)=0$.   The well-known result of
 Spitzer \cite{sp} states  the  convergence of $2\theta(t)/ \log t$ in law:
\begin{align*}
\lim_{t \to \infty}P\bigg(\frac{2\theta(t)}{\log t}\le a \bigg)=\frac{1}{\pi}\int_{-\infty}^a\frac{dx}{1+x^2}. 
\end{align*}
It is shown in \cite{werner}  that for any increasing function $f:(0,\infty)\to(0,\infty)$
\begin{align}\label{oo}
\limsup_{t\to \infty} \frac{\theta(t)}{f(t)}  =0 \text{ or }\infty \quad \text{ a.s.}
\end{align}
according as the integral $\int^{\infty}\frac{1}{f(t)t}dt$ converges or diverges  and 
\begin{align*}
\liminf_{t\to \infty} \frac{1}{f(t)}\sup\{\theta(s),1\le s\le t\} =0 \text{ or }\infty \quad \text{ a.s.}
\end{align*}
according as the integral $\int^{\infty}\frac{f(t)}{t(\log t)^2}dt$  diverges or converges;
moreover, it is shown that the square root of the random time 
$H(t)$ is subjected to the same   growth law as  of $\theta$ in (\ref{oo}) and the $\liminf$ behavior of $H(t)$ is also given.    
 Another proof of  (\ref{oo}) is  given in \cite{shi}.
Also, it is shown in \cite{s}
\begin{align*}
\liminf_{t\to \infty} \frac{\log \log \log t}{\log t}\sup\{|\theta(s)|,1\le s\le t\} =\frac{\pi}{4} \quad \text{ a.s..}
\end{align*}

Before advancing  our result we recall the two arcsine laws whose analogues are studied in this paper.
Let  $\{B(t): t\ge0\}$  be  a standard linear Brownian motion started at zero and denote by $Z_t$ the time when the maximum of $B_s$ in the interval  $0\le s\le t$ is attained.  Then, 
the process  $Z_t$  and the process  $\sup\{s\in [0,t]: B(s)=0\}$, the last zero of Brownian motion in the time interval  $[0,t]$,  are subject to  the same law, and  
according to L\'{e}vy's arcsine law the scaled variable  $Z_t/t$ is subject to the arcsin law.  
 (cf. e.g., \cite{Mo} Theorem $5.26$ and $5.28$)

In order to state the results of this paper we set 
\begin{align}
V(a)=\frac{4}{\pi^2}\int\!\!\!\int_{0\le{y}\le a x}\frac{dx}{1+x^2}\frac{dy}{1+y^2}.
\end{align}
We also define a random variable $M_t\in[0,t]$ by
\begin{align*}
\theta(M_t)=\max_{s\in [0,t]}\theta(s),
\end{align*}
the time when $\theta(s)$ attains  the maximum  in the interval  $0\le s \le t$, 
and a random variable $L_t$ by 
 \begin{align*}
L_t=\sup\{s\in [0,t]: \theta(s)=0\},
\end{align*}
the last zero of   $\theta(s)$ in $[0,t]$.
According to Theorem $2.11$  of \cite{Mo} a linear Brownian motion attains its maximum at a single point  on each  finite interval with probability one. 
In view of the representation  $\theta(t)=B(H(t))$, it therefore follows that the maximiser $M_t$ is uniquely determined for all  $t$  with probability one. 
\begin{thm}\label{m1}
(a) For every  $0<a<1$
\begin{align*}
\lim_{t\to \infty} P\bigg(\frac{\log M_t}{\log t}\le a\bigg)=V\bigg(\frac{a}{1-a}\bigg).
\end{align*}

(b)  It holds that
\begin{align*}
\{L_t: t \ge 0 \}=_d \{M_t: t \ge 0 \}.
\end{align*}
\end{thm}

\begin{thm}\label{ili}
Let $\alpha(t)$ be a  positive  function  that is non-increasing, tends to zero  as $ t\to \infty$  and satisfies 
\begin{align}\label{assu}
2\alpha(t^e) \ge \alpha(t),
\end{align}
and put
\begin{align*}
I\{\alpha\}={\displaystyle \int^{\infty}\frac{\alpha(t) |\log \alpha(t)|}{t\log t}dt} .
\end{align*}
 Then, with probability one 
\begin{align*}
\liminf_{t\to \infty} \frac{M_t}{t^{\alpha(t)}}=\infty ~ \text{ or } ~  0 \quad 
\end{align*}
according as the integral 
$I\{\alpha\}$ converges  or diverges.
\end{thm}

It may be worth noting that the distribution function  $V(a/(1-a))$ ($0\le a \le 1)$ is expressed as
 $$V\bigg(\frac{a}{1-a}\bigg) =\int_0^a \frac1{2u-1} \log \frac{u}{1-u} du.$$  
 Indeed, 
  $$V'(c)= \int_0^\infty \frac{x dx}{(1+x^2)(1+c^2 x^2)} = \frac{\log c}{c^2-1} ~~~(c\neq 1),$$
  where 
  \[
  \frac{d}{da}V\bigg(\frac{a}{1-a}\bigg) = \frac{1}{(1-a)^2}V'\bigg(\frac{a}{1-a}\bigg) \quad
(a\neq \frac{1}{2}),
  \]  
  and we find the density asserted above.

\section{Proofs}
\subsection{Proof of Theorem \ref{m1}}
Let $\{N(t): t\ge0\}$ be the maximum process of a winding number $\{ \theta(t): t\ge0\}$, 
i.e. the process defined by
\begin{align*}
N(t)=\max_{s\in[0,t]} \theta(s).
\end{align*}
\begin{lem}\label{ref2}
If $a>0$, then $P(N(t)>a)=2P(\theta(t)>a)=P(|\theta(t)|>a)$.
\end{lem}
\begin{proof}
By reflection principle \cite{Mo}, (Theorem $2.21$) it holds that for any $t>0$
\begin{align*}
\max_{0\le l \le t}B(l)=_d|B(t)|.
\end{align*}
By Skew-product representation $B(t)$ is independent of $|W(t)|$,  
hence since  $B(l)$ is independent of $H(t)=\int_0^{t} \frac{dm}{|W(m)|^2}$, 
it holds
\begin{align*}
\max_{0\le l \le t}B(H(l))=_d|B(H(t))|,
\end{align*}
showing the assertion of the lemma.
\end{proof}
\begin{lem}\label{l1}
$\{N(t)-\theta(t): t\ge0\}=_d \{|\theta(t)|: t\ge0 \}$.
\end{lem}
\begin{proof}
According to L\'{e}vy's representation of the reflecting Brownian motion \cite{Mo}, (Theorem $2.34$) we have 
\begin{align*}
\{\max_{0\le l \le t}B(l)-B(t):t\ge0\}=_d\{|B(t)|:t\ge0\}.
\end{align*}
Hence as in the preceding proof,
\begin{align*}
\{\max_{0\le l \le t}B(H(l))-B(H(t)):t\ge0\}=_d\{|B(H(t))|:t\ge0\},
\end{align*}
as desired.
\end{proof}

\begin{proof}[Proof of Theorem \ref{m1}]
Lemma \ref{l1} together with Lemma \ref{ref2} show that  the process $\{M_s: s \ge 0 \}$ has the same law as $\{L_s: s \ge 0 \}$, being nothing but the last zero of the process $\{N(t)-\theta(t):  0\le t\le s \}$ for any $s$. 
So it remains to prove part (a). 
Fix $a\in(0,1)$. 
Set $T_c=\inf\{l\ge0: |W(l)|=c\}$, for which we sometimes write $T(c)$ for typographical reasons. 
We first prove the upper bound. 
By (\ref{skew}) it holds that
\begin{align}\notag
P(M_t<t^a)
=&P(\max_{0\le u \le t^a}B(H(u))>\max_{t^a \le u\le t} B(H(u)))\\
\notag
=&P(\max_{0\le u\le t^a}B(H(u))-B(H(t^a))>\max_{t^a\le u \le t}B(H(u))-B(H(t^a)))\\
\label{again}
=&P(\max_{0\le u\le t^a}B(H(u))-B(H(t^a))>\max_{t^a\le u \le t}\tilde{B}(H(u))-\tilde{B}(H(t^a))),
\end{align}
where $\tilde{B}$ is a linear Brownian motion started at zero which is independent of $W$. 
Corresponding to (\ref{skew}) we can write 
 $\tilde{W}(0)=(1,0)$, $\arg \tilde{W}(l)=\tilde{B}(\tilde{H}(l))$, 
$\tilde{H}(l)=\int_0^l \frac{dm}{|\tilde{W}(m)|^2}$ with $\tilde{W}$ independent of $W$,  and put $\tilde{T}_c=\inf\{l\ge0: |\tilde{W}(l)|=c\}$. 
By Lemma \ref{ref2} and Lemma \ref{l1} we have  
$\max_{0\le u\le t^a}B(H(u))-B(H(t^a))=_d\max_{0\le u\le t^a}B(H(u))$, 
and therefore
\begin{align}\notag
&P(\max_{0\le u\le t^a}B(H(u))-B(H(t^a))>\max_{t^a\le u \le t}\tilde{B}(H(u))-\tilde{B}(H(t^a)))\\
\label{again1}
=&P(\max_{0\le u\le t^a}B(H(u))>\max_{t^a\le u \le t}\tilde{B}(H(u))-\tilde{B}(H(t^a))).
\end{align}
By standard large deviation result (cf. e.g., \cite{Law3}, $(11)$ and $(12)$), given $\epsilon>0$, it holds that for all sufficiently large $t$
\begin{align*}
P(t^a \le T_{t^{\frac{a+\epsilon}{2}}}, 
T_{t^{\frac{1-\epsilon}{2}}} \le t)\ge 1-\epsilon. 
\end{align*}
Therefore, we get
\begin{align}\notag
&P(\max_{0\le u\le t^a}B(H(u))>\max_{t^a\le u \le t}\tilde{B}(H(u))-\tilde{B}(H(t^a)))\\
\label{smpp}
\le &P(\max_{0\le u\le T(t^{\frac{a+\epsilon}{2}} )}B(H(u))>\max_{T(t^{\frac{a+\epsilon}{2}})\le u \le T(t^{\frac{1-\epsilon}{2}})}\tilde{B}(H(u))-\tilde{B}(H(T_{t^{\frac{a+\epsilon}{2}}})))+\epsilon.
\end{align}
Also, strong Markov property tells us
\begin{align*}
&\int_{T_{t^{\frac{a+\epsilon}{2}}}}^{T_{t^{\frac{1-\epsilon}{2}}}}\frac{dm}{|W(m)|^2}
=_d\int_0^{\tilde{T}_{t^{\frac{1-a-2\epsilon}{2}}}}\frac{dm}{|\tilde{W}(m)|^2},\\
\text{and  }& H(T_{t^{\frac{1-\epsilon}{2}}})-H( T_{t^{\frac{a+\epsilon}{2}}})
\text{ is independent of } H(T_{t^{\frac{a+\epsilon}{2}}}).
\end{align*}
So, if we set for $a,b<\infty$
\begin{align*}
Q(a,b)=P(\max_{0\le u\le T(a)}B(H(u))>\max_{0\le u \le \tilde{T}(b)}\tilde{B}(\tilde{H}(u))),
\end{align*}
it holds that
\begin{align}\label{smp}
P(\max_{0\le u\le T(t^{\frac{a+\epsilon}{2}} )}B(H(u))>\max_{T(t^{\frac{a+\epsilon}{2}})\le u \le T(t^{\frac{1-\epsilon}{2}})}\tilde{B}(H(u))-\tilde{B}(H(T_{t^{\frac{a+\epsilon}{2}}})))
= Q(t^{\frac{a+\epsilon}{2}},t^{\frac{1-a-2\epsilon}{2}} ).
\end{align}
Note that by Skew-product representation $B(t)$( resp. $\tilde{B}(t)$) is independent of $H(T_{t^{\frac{a+\epsilon}{2}}})$( resp. $\tilde{H}(\tilde{T}_{t^{\frac{a+\epsilon}{2}} })$). 
Then, if $\tilde{\theta}(l)=\tilde{B}(\tilde{H}(l))$, by reflection principle we get
\begin{align}\notag
Q(t^{\frac{a+\epsilon}{2}},t^{\frac{1-a-2\epsilon}{2}} )
&=P(|B(H(T_{t^{\frac{a+\epsilon}{2}} }))|>|\tilde{B}(\tilde{H}(\tilde{T}_{t^{\frac{1-a-2\epsilon}{2}}}))|)\\
\label{q1}
&=P(|\theta(T_{t^{\frac{a+\epsilon}{2}}})|>|\tilde{\theta}(\tilde{T}_{t^{\frac{1-a-2\epsilon}{2}}})|).
\end{align}
Moreover, since  $\theta(T_r)$ follows the Cauchy distribution with parameter $|\log r|$ 
(cf. e.g., \cite{yor}, Section $5$, Exercise $2.16$, 
\cite{v}, Proposition $2.3$, and
\cite{w}
), we get
\begin{align}\label{q2}
Q(t^{\frac{a+\epsilon}{2}},t^{\frac{1-a-2\epsilon}{2}} )
=P(|\theta(T_{t^{\frac{a+\epsilon}{2}}})|>|\tilde{\theta}(\tilde{T}_{t^{\frac{1-a-2\epsilon}{2}}})|)
=V(\frac{a+\epsilon}{1-a-2\epsilon}).
\end{align}
Therefore, since $\epsilon$ is arbitrary, this gives the desired upper bound. 

Next, we prove the lower bound. 
By standard large deviation result (cf. e.g., \cite{Law3}, $(11)$ and $(12)$), 
given $\epsilon>0$, it holds that 
for all sufficiently large $t$
\begin{align}
P(T_{t^{\frac{a-\epsilon}{2}}}\le t^a,
t\le T_{t^{\frac{1+\epsilon}{2}}} )\ge 1-\epsilon.
\end{align}
Moreover, by repeating the argument in (\ref{smpp}) and (\ref{smp}), we get
\begin{align*}
&P(\max_{0\le u\le t^a}B(H(u))>\max_{t^a\le u \le t}\tilde{B}(H(u))-\tilde{B}(H(t^a)))\\
\ge &Q(t^{\frac{a-\epsilon}{2}},t^{\frac{1-a+2\epsilon}{2}})-\epsilon.
\end{align*}
Therefore, repeating the arguments in (\ref{again}), (\ref{again1}), (\ref{q1}) and (\ref{q2}), we get
\begin{align*}
P(M_t<t^a)=
&P(\max_{0\le u\le t^a}B(H(u))>\max_{t^a\le u \le t}\tilde{B}(H(u))-\tilde{B}(H(t^a)))\\
\ge &Q(t^{\frac{a-\epsilon}{2}},t^{\frac{1-a+2\epsilon}{2}})-\epsilon\\
=&V(\frac{a-\epsilon}{1-a+2\epsilon})-\epsilon,
\end{align*}
yielding the lower bound.
\end{proof}
\subsection{Proof of Theorem \ref{ili}}
\begin{proof}[Proof of Theorem \ref{ili}]
We first prove $\liminf_{t\to \infty} M_t/t^{\alpha(t)}=
\infty$ if $I\{\alpha\}<\infty$. 
We may replace $\alpha(t)$ by $\alpha(t) \lor (\log \log t)^{-2}$. 
Indeed, if we set $$\tilde{\alpha }(t)=\alpha(t)1\{\alpha(t)>(\log \log t)^{-2}\}+(\log \log t)^{-2}1\{\alpha(t)\le(\log \log t)^{-2}\},$$  $I\{\tilde{\alpha}\}<\infty$. 
By standard large deviation result (cf. e.g., \cite{Law3}, $(11)$ and $(12)$)  for any $q<\infty$ there exist $0<c_1$, $c_2<\infty$ such that
\begin{align}\label{k1}
P(qt^{4\alpha(t)}\le T(t^{4\alpha(t)}), 
T(t^{\frac{1}{2}-\alpha(t)})\le t)\ge 1-c_1\exp(-t^{c_2\alpha(t)}). 
\end{align}
Therefore, by the same arguments as made for (\ref{again}), (\ref{again1}),  (\ref{smpp}), (\ref{smp}), (\ref{q1}) and  (\ref{q2}) we infer that for any $q<\infty$
\begin{align*}
P(M_t<qt^{4\alpha(t)})
=&P(\max_{0\le u\le qt^{4\alpha(t)} }B(H(u))-B(H(qt^{4\alpha(t)}))>\max_{qt^{4\alpha(t)}\le u \le t}\tilde{B}(H(u))-\tilde{B}(H(qt^{4\alpha(t)})))\\
\le &Q(t^{4\alpha(t)},t^{\frac{1}{2}-5\alpha(t)})+c_1\exp(-t^{c_2\alpha(t)})\\
=&V( \frac{4\alpha(t)}{\frac{1}{2}-5\alpha(t)})+c_1\exp(-t^{c_2\alpha(t)}).
\end{align*}
We set $t_n=\exp(e^n)$. 
Then, noting that $V(\alpha(n))\asymp \alpha(n)|\log \alpha(n)|$, we deduce from (\ref{k1}) that  for some $C<\infty$
\begin{align*}
P(M_{t_n}< t_n^{4\alpha(t_n)})
\le C\alpha(t_n)|\log \alpha(t_n)|+c_1\exp(-t_n^{c_2\alpha(t_n)}).
\end{align*}
The sum of the right-hand side over $n$ is finite 
since $\sum_{n=1}^{\infty}\alpha(t_n)|\log \alpha(t_n)|<\infty$ 
if $I\{\alpha\} <\infty$, 
and $\alpha(t)\ge (\log \log t)^{-2}$ 
according to our assumption. 
Thus, by Borel-Cantelli lemma for any $q<\infty$, with probability one
\begin{align}\label{p1}
\frac{M_{t_n}}{t_n^{4\alpha(t_n)}}>q\quad \text{for almost all }n.
\end{align}
Note that if we choose $t$ such that $t_n<t\le t_{n+1}$, 
then $t_n^{4\alpha(t_n)}>t^{\alpha(t)}$ 
and from (\ref{p1}) it follows that $M_t>M_{t_n}>qt^{\alpha(t)}$ for all sufficiently large $n$. 
Hence,
\begin{align*}
\liminf_{t\to \infty} \frac{M_t}{t^{\alpha(t)}}>q\quad a.s..
\end{align*}
Since $q<\infty$ is arbitrary, this concludes the proof. 

Next, we prove  $\liminf_{t\to \infty} M_t/t^{\alpha(t)}=0$ assuming that $I\{\alpha\}=\infty$. 
For any $a<b<\infty$, we set
\begin{align*}
\theta^*[a,b]=\max \{ \theta(t): T_a \le t \le T_b \},
\end{align*}
and define $\overline{M}[a,b]$ via 
\begin{align*}
\theta(\overline{M}[a,b])=\theta^*[a,b]
\quad \text{ and }T_a \le \overline{M}[a,b] \le T_b.
\end{align*}
Recall we have set $t_n=\exp(e^n)$. 
For $q>0$, denote by $A_n$ the event  
$$\overline{M}[qt_n^{\alpha(t_n)}, t_n]<T(qt_n^{2\alpha(t_n)}).$$ 
Bringing in the set $D=\{n \in \mathbb{N} :\alpha(t_n)>\frac{1}{(\log \log t_n)^2}\}$, we shall prove $\sum_{n=1, n\in D}^{\infty} P(A_n)=\infty$ and 
\begin{align}\label{jj}
\liminf_{n\in D,n \to \infty} \frac{\sum_{j=1,j \in D}^n \sum_{k=1,k \in D}^n P(A_j \cap A_k)}
{(\sum_{j=1, j\in D}^n P(A_j))^2}<\infty,
\end{align}
which together imply 
$P(\limsup_{n\in D,n \to \infty} A_n )=1$ 
according to the Borel-Cantelli lemma (cf. \cite{sp1}, p.$319$ or \cite{lam}) and  Kolmogorov's $0-1$ law. 
First we prove $\sum_{n=1,n \in D}^{\infty} P(A_n)=\infty$. 
Note that it holds that for $0<a<b<c$
\begin{align*}
P(\theta^*[a,b]>\theta^*[b,c])
=P(\theta^*[1,\frac{b}{a}]>\theta^*[\frac{b}{a},\frac{c}{a}]).
\end{align*}
Thus,
\begin{align*}
P(\theta^*[qt^{\alpha(t)},qt^{2\alpha(t)}]>\theta^*[qt^{2\alpha(t)},t])
=P(\theta^*[1,t^{\alpha(t)}]>\theta^*[t^{\alpha(t)},\frac{1}{q}t^{1-\alpha(t)}]).
\end{align*}
Therefore, we get by the same argument as employed for (\ref{again}), (\ref{again1}),  (\ref{smpp}), (\ref{smp}), (\ref{q1}) and (\ref{q2})
\begin{align}\notag
&P(\overline{M}[qt^{\alpha(t)}, t]<T(qt^{2\alpha(t)}))\\
\notag
=&P(\theta^*[1,t^{\alpha(t)}]>\theta^*[t^{\alpha(t)},\frac{1}{q}t^{1-\alpha(t)}])\\
\notag
=&P(\max_{ u \le T(t^{\alpha(t)})}B(H(u))-B(H(T(t^{\alpha(t)})))>\max_{T(t^{\alpha(t)})\le u\le 
T(\frac{1}{q}t^{1-\alpha(t)})} \tilde{B}(H(u))-\tilde{B}(H(T(t^{\alpha(t)}))))\\
\notag
= &Q(t^{\alpha(t)},\frac{1}{q}t^{1-2\alpha(t)})\\
\label{t1}
=&V(\frac{\alpha(t)}{1-2\alpha(t)-(\log t \log q)^{-1}}).
\end{align}
Moreover, using $V(\alpha(n))\asymp \alpha(n)|\log \alpha(n)|$ again, 
we get for some $C>0$
\begin{align*}
P(A_n)
\ge C\alpha(t_n)|\log \alpha(t_n)|.
\end{align*} 
It holds that $\sum_{n\in D}\alpha(t_n)|\log \alpha(t_n)|= \infty$ 
 if $I\{\alpha\}=\infty$, 
since 
 $\sum_{n\notin D}\alpha(t_n)|\log \alpha(t_n)|< \infty$. 
So we get $\sum_{n\in D} P(A_n)=\infty$. 

Next we prove (\ref{jj}).  
We only need to consider $\sum_{j=1,j \in D} \sum_{k<j, k \in D} P(A_j \cap A_k)$. 
First we consider $\sum_{j=1,j \in D}^n \sum_{k \in R_{k,j} ,k \in D} P(A_j \cap A_k)$ where $R_{k,j}=\{k: qt_j^{\alpha(t_j)}\ge  t_k\}$.  
Note that for  $a<b\le c<d<\infty$
\begin{align}
\label{t2}
\overline{M}[a,b]-T_a \text{ is independent of } \overline{M}[c,d]-T_c.
\end{align}
Then, since $qt_k^{\alpha(t_k)}<t_k \le qt_j^{\alpha(t_j)}< t_j$ 
when $k$ is satisfied with $ qt_j^{\alpha(t_j)}\ge  t_k$, 
it holds that 
\begin{align}\label{ind}
P(A_j \cap A_k)=P(A_j)P(A_k).
\end{align} 
So, next we consider the case $qt_j^{\alpha(t_j)}< t_k$. 
We denote by $A'_{k,j}$ 
the event $\overline{M}[qt_k^{\alpha(t_k)},qt_j^{\alpha(t_j)}]<T(qt_k^{2\alpha(t_k)})$. 
Note that when $k$ is satisfied with $qt_j^{\alpha(t_j)}< t_k$, 
we have  
 $A_k \subset A'_{k,j}$, and by (\ref{t2}) $P(A_j\cap A'_{k,j})=P(A_j)P(A'_{k,j})$. 
Then, since by the same argument for (\ref{t1}) $P(A'_{k,j})=V(\frac{e^k\alpha(t_k)}{e^j\alpha(t_j)-e^k\alpha(t_k)})$,
 we get
\begin{align}\label{x1}
P(A_j \cap A_k)\le P(A_j\cap A'_{k,j})=P(A_j)P(A'_{k,j})=P(A_j) V(\frac{e^k\alpha(t_k)}{e^j\alpha(t_j)-e^k\alpha(t_k)}).
\end{align}
Furthermore, since $\alpha(t_k)\le 2\alpha(t_{k+1})$ due to the assumption (\ref{assu}), we get
\begin{align}\notag
&\sum_{ k\in R_{k,j}^c,k<j, k \in D} P(A'_{k,j})
=\sum_{ k \in R_{k,j}^c,k<j,k \in D} V(\frac{e^k\alpha(t_k)}{e^j\alpha(t_j)-e^k\alpha(t_k)})\\
\label{x2}
&\le\sum_{k=1}^{\infty} V(\frac{2^k}{e^k-2^k})
\le C\sum_{k=1}^{\infty}{(\frac{e}{2})}^{-k}\le C',
\end{align}
where $R_{k,j}^c=\{k: qt_j^{\alpha(t_j)}<  t_k\}$. 
So, by (\ref{x1}) and (\ref{x2}) we get 
$\sum_{j=1,j \in D}^n \sum_{k \in R_{k,j}^c,k \in D} P(A_j \cap A_k)
\le C \sum_{j=1,j \in D}^nP(A_j)$.  
Combined with (\ref{ind}) this shows 
$$\sum_{j=1,j \in D}^n \sum_{k\le j,k\in D}^n P(A_j \cap A_k) \le \sum_{j=1,j\in D}^n \sum_{k\le j,k\in D}^n P(A_j) P(A_k)+C'\sum_{j=1,j\in D}^nP(A_j) ,$$
 completing the proof  of (\ref{jj}). 
Therefore, we can conclude that with probability one
\begin{align}\label{o1}
\overline{M}[qt_n^{\alpha(t_n)}, t_n]<T(qt_n^{2\alpha(t_n)})\quad \text{ infinitely often for }n \in D.
\end{align}
On the other hand, by standard large deviation result (cf. e.g., \cite{Law3}, $(11)$ and $(12)$) 
there exist $0<c_3$, $c_4<\infty$ such that
\begin{align*}
P( T(qt^{2\alpha(t)})\le qt^{5\alpha(t)}, 
t^{\frac{1}{4}} \le T_t)\ge 1-c_3\exp(-c_4t^{\alpha(t)}). 
\end{align*}
Moreover, $\sum_{n\in D} c_3\exp(-c_4t_n^{\alpha(t_n)})<\infty$. 
Then, by Borel-Cantelli lemma it holds that with probability one
\begin{align}\label{o2}
T(qt_n^{2\alpha(t_n)})\le  qt_n^{5\alpha(t_n)}, \quad
M_{t_n^{\frac{1}{4}}}  \le \overline{M}[qt_n^{\alpha(t_n)}, t_n],
\quad \text{for almost all }n\in D.
\end{align}
So, by (\ref{o1}) and (\ref{o2}) it holds that
\begin{align*}
\liminf_{t\to \infty} \frac{M_t}{qt^{20\alpha(t)}}
\le \liminf_{n \in D,n\to \infty} \frac{M_{t_n}}{qt_n^{20\alpha(t_n)}}
\le \liminf_{n \in D,n\to \infty} \frac{M_{t_n^{\frac{1}{4}}} }{qt_n^{5\alpha(t_n)}}
\le \liminf_{n \in D,n\to \infty} \frac{\overline{M}[qt_n^{\alpha(t_n)}, t_n]}{T(qt_n^{2\alpha(t_n)})}
<1 \quad a.s..
\end{align*}
The proof finishes since $q>0$ is arbitrary by replacing $\alpha(t)$ by $\frac{\alpha(t)}{20}$. 
\end{proof}

\end{document}